\documentclass[11pt,a4paper]{article}

\usepackage[body={154mm,229mm},top=34mm]{geometry}
\usepackage{bm}
\usepackage{amsmath}
\usepackage{amssymb}
\usepackage{amsthm}
\usepackage{amscd}
\usepackage[latin1]{inputenc}
\usepackage{setspace}
\usepackage{subcaption}
\usepackage{comment, url, pifont}
\usepackage{dsfont}
\usepackage{graphicx}
\usepackage{color}
\usepackage{hyperref}
\usepackage{mathrsfs} 

\font\dsrom=dsrom10 scaled 1200
\def \indic{\textrm{\dsrom{1}}}

\newtheorem{lemma}{Lemma}
\newtheorem{corollary}{Corollary}

\newtheorem{remark}{Remark}
\newtheorem{proposition}{Proposition}

\title{Stationary analysis of a single queue with remaining service time dependent arrivals}

\author{ Benjamin Legros$^1$ $\bullet$ Ali Devin Sezer$^2$ \\
~\\
$^1${\em Laboratoire Genie Industriel, CentraleSup\'{e}lec}
{\em Universit\'{e} Paris-Saclay}\\
{\em Grande Voie des Vignes, 92290 Chatenay-Malabry, France}
~\\
\\
{\em $^2$ Middle East Technical University}
{\em Institute of Applied Mathematics}\\
{\em Ankara, 06800, Turkey}
\\
 {\small \tt benjamin.legros@centraliens.net} $\bullet$ {\small \tt devin@metu.edu.tr}
}

\begin{document}
\maketitle

\begin{abstract}
We study a generalization of the $M/G/1$ system
(denoted by $rM/G/1$) 
with independent and identically distributed 
(iid) service times and with an arrival process whose arrival rate
$\lambda_0f(r)$ depends
on the remaining service time $r$ of the current customer being served.
We derive a natural stability condition and
provide a stationary analysis
under it both at service completion
times (of the queue length process) 
and in continuous time 
(of the queue length and the residual service time).
In particular, we show that the stationary
measure of queue length at service completion times is equal to that of
a corresponding $M/G/1$ system. For $f > 0$ we show that the continuous
time stationary measure of the $rM/G/1$ system is linked to the $M/G/1$
system via a time change.
As opposed to the $M/G/1$ queue, the stationary measure of queue length
of the $rM/G/1$ system at service completions differs from its marginal distribution
under the continuous time stationary measure. Thus, in general,
arrivals of the $rM/G/1$ system do not see time averages.
We derive formulas for the average queue
length, probability of an empty system 
and average waiting time under the continuous time stationary measure. 
We provide examples showing the effect of changing the reshaping function on 
the average waiting time.

\end{abstract}

\section{Introduction}
The goal of the present note is the steady state analysis
of a single server queueing system with iid service times and
an arrival process whose rate is a function of the remaining
service time of the current customer being served, if the server is busy,
or a constant $\lambda_0$ otherwise. This is a generalization of the
$M/G/1$ system. Because the arrival rate is allowed to depend on the remaining service
time we will denote it by the notation $`rM/G/1.'$
Arrival processes with remaining service time dependent rates
can be used to model systems where customers can directly estimate 
the remaining service time by observing the amount of work that a server 
has to treat and use this information to decide whether to join the queue or not.
This type of behavior occurs, for example, at
checkout queues in supermarkets. A potential application area for $rM/\cdot/\cdot$ systems
is call centers \cite{aksin07,gans03} with inbound and outbound calls. 
Modern call centers call out customers to connect them with a server even when all servers are busy \cite{pang14}; 
the decision to initiate an outbound call can use estimates of the remaining service time of the busy servers.
New approaches to call center modeling also allow the control of the arrival process of inbound calls by postponing their
routing to an agent or by giving incentives to callback later \cite{legros16}; such approaches can
make use of estimates of the remaining service time of servers. 
Generalizations of the $rM/G/1$ model may be useful in the analysis of these systems.

Queues with queue-length dependent and Markov modulated arrival or 
service time distributions have been studied in the literature, see, e.g.,
 \cite{dshalalow1991single,bekker2004queues,kumar2009enhanced,perry2013duality}.
The only works we are aware of allowing the arrival rate to depend
on the remaining service time are 
\cite{knessl1987busy,knessl1987markov,knessl1987markov2}; these works
study the remaining service time process (denoted by $U(t)$ in these works)
when the arrival rate and the service rate of the arriving customer depends 
on $U$ (\cite{knessl1987markov, knessl1987markov2}
further contain two state Markov modulation whose transition rates
depend on $U$). The analysis method used in these works is asymptotic
approximation as arrival, service and transition rates are scaled by a
parameter whose value is sent to $\infty.$ In the current work, we study,
within a narrower framework,
the joint queue length and remaining service time distribution and
our focus is on finding exact solutions.

To simplify exposition, 
we assume that the iid service times have a density,
denoted by $g(\cdot)$. We further comment on this assumption in Section \ref{s:conclusion}.
The arrival process of customers is Poisson with
constant arrival rate $\lambda_0$ if the system is empty or $\lambda_0 f(r)$ if the server is busy and 
the remaining service time of the customer being served is $r$.
In the particular case where $f(r)=1$ for $r\geq 0$, the system reduces to an $M/G/1$ queue. $f$ can be interpreted
in two ways: if $f(r) \in (0,1)$, $r \in {\mathbb R}_+$,
then $f(r)$ can be thought of as the probability 
that an arriving customer 
joins the queue after having observed the remaining service time $r$.
$f$ can also be thought of as a control parameter that transforms / reshapes, 
the constant arrival rate $\lambda_0$ to optimize system performance.
With this interpretation in mind, we will refer to $f$ as the ``reshaping function'' 
(the `$r$' in the abbreviation $rM/G/1$ refers also to ``reshaping'' of the arrival process).
For the latter interpretation, a natural condition on a reshape function is that it doesn't
change the overall average arrival rate to the system. 
In Proposition 
\ref{p:arrivals} of Subsection \ref{ss:averagear} the average arrival rate 
to an $rM/G/1$ system is computed to be 
$\alpha = \frac{\lambda_0}{1-\lambda_0(\bar{\nu} - \nu)}$, where $\nu = \int_0^\infty r g(r)dr$
is the average service length and $\bar{\nu} = \int_0^\infty F(r) g(r)dr$ with $F(r) = \int_0^r f(u)du.$ 
Thus, under the assumption
\begin{equation}\label{e:l0unchanged}
\nu = \bar{\nu}
\end{equation}
the average arrival rate of an $rM/G/1$ system remains $\lambda_0.$ 
This assumption will be in force in Section \ref{s:example} where we compare 
the average waiting times of a range of $rM/G/1$ systems with the same service time distribution
and average arrival rate $\lambda_0$ but different reshape functions.

A natural framework for the study of the $rM/G/1$ queue is the piecewise deterministic processes (PDP) 
of \cite{davis1993markov}.
Section \ref{s:construct} gives a construction of the $rM/G/1$ process as a
piecewise deterministic
Markov process based on this framework. The process is 
$X_t=(N_t,R_t)$; its first component represents
the number of customers 
(i.e., queue length, including the customer being served) 
in the system the second component represents the remaining service time.
Subsection \ref{ss:generator} gives its generator and Subsection \ref{ss:completion} derives 
the dynamics of the embedded random walk ${\mathcal N}$, which is the sequence of queue lengths observed at
service completion times; Proposition \ref{p:MG1reduce} shows that  
the dynamics of ${\mathcal N}$ equals
that of the embedded random walk (at service completion times) 
of an $M/G/1$ queue (whose state
process is denoted by $\bar{X}$) with constant arrival 
rate $\lambda_0$ and with iid service times $\{\bar{\sigma}_k, k=1,2,3,\}$
where $\bar{\sigma}_k = F(\sigma_k)$, and $\{\sigma_k\}$ are the iid service times of the original $rM/G/1$ 
system. The stationary distribution of the $rM/G/1$ system at service completions (and arrivals) follows
from this reduction; the details are given in Section \ref{s:servcomp}. 
Proposition \ref{p:invariantmeasure}
derives the stability condition $\rho \doteq \lambda_0 \bar{\nu} < 1$, \eqref{e:pk1} gives the
expected stationary queue length at service completions and \eqref{e:mgf} gives the stationary moment generating
function of the queue length distribution at service completions.

As opposed to $M/G/1$ queues, the stationary distribution of queue length 
of an $rM/G/1$ system in continuous time does not
equal its stationary distribution at service completions; 
therefore, for $rM/G/1$ queues, the continuous time stationary distribution 
and service measures based on it must be computed directly. 
Section \ref{s:stationarycont} begins with
the statement and recursive solution of the balance equation for the 
stationary distribution of the continuous time process $X$, which 
consists essentially of a sequence of linear ordinary differential
equations (ODE) where $f$ 
serves as an $r$ dependent coefficient.
Proposition \ref{p:proofstationary} proves that the solution of the balance equation is indeed the stationary measure of the process $X$ under the stability 
assumption $\rho < 1.$ The proof is based
on the PDP framework of \cite{davis1993markov}. A number of further computations based on the
continuous time stationary distribution is given in Section \ref{s:stationarycont}; 
in particular, Corollary \ref{c:emptysystemprob} gives
a simple formula for the stationary probability of an empty $rM/G/1$ system in continuous time
and Proposition \ref{p:meanformula} gives a formula for the stationary expected queue length in continuous time. Proposition \ref{p:arrivals} of 
Subsection  \ref{ss:averagear} gives the average arrival rate for the $rM/G/1$ system and 
finally \eqref{e:averagesojourntime} gives an explicit formula for the average 
sojourn time of a customer in an $rM/G/1$ system.
In general, $f$ may take the value $0$ and this may make $F$ noninvertible. For this reason, there is not, in general, a bijective correspondence between the 
continuous time stationary distribution of the $rM/G/1$ process $X$ and that 
of the $M/G/1$ process $\bar{X}.$ However, for $f > 0$ a bijective 
correspondence can be established; this is treated in Subsection \ref{ss:MG1c2}.

Section \ref{s:example} gives two examples showing the impact of reshaping 
the arrival process on the average waiting time. We
observe, as expected, that, for a given average arrival rate, 
the closer the customers arrive to the end of a service the shorter will be 
the average waiting time in the system.
Section \ref{s:conclusion} points out directions for future 
research.

\section{Dynamics of the process}\label{s:construct}
The theory of piecewise-deterministic Markov Processes (PDP) of \cite{davis1993markov} 
provides the ideal mathematical framework for the analysis of the $rM/G/1$ 
queue.
For the definition of the process we will use the PDP definition given in \cite[page 57]{davis1993markov},
which uses the following elements 
(all adopted from \cite{davis1993markov}):
the state space of the process will be
\[
E \doteq \bigcup_{k=0}^\infty E_k,
E_0 = B({\bm 0},\delta) \subset {\mathbb R}^2,~~
E_k \doteq \{k \} \times {\mathbb R}_+ = \{ (k,r), r  > 0 \},~~~
k\in \{ 1,2,3,...\},
\]
where ${\bm 0} =(0,0) \in {\mathbb R}^2$ denotes the origin of
${\mathbb R}^2$ and  $B({\bm 0}, \delta)$ 
denotes an open ball of radius $\delta < 1$;
${\bm 0}$ represents
the empty system
(in \cite{davis1993markov} the letter $\zeta$ denotes the
second component of  $x \in E$, we use $r$ for the same purpose).
The $rM/G/1$ process, $X_t=(N_t,R_t) \in E$,  $t \ge 0$,
will evolve,
on each $E_k$
smoothly following the vector field
$\mathfrak{X}_k : E_k \mapsto {\mathbb R}^2$
\[
\mathfrak{X}_k(x) \doteq \begin{cases} (0,-1),& ~~ k > 0, \\
					  0, & \text{otherwise,} 
\end{cases}
\] until it jumps.
Let us denote the jump times
of $X$ by the sequence $\{T_i, i=1,2,3,...\}.$
The vector field $\mathfrak{X}_k$ defines the following trivial flow
\begin{equation}\label{e:detdynamics}
\phi(t,(k,r))= (k,r - t),~~ k > 0,~~~~~~ \phi(t,{\bm 0} ) = {\bm 0};
\end{equation}
the process $X$ follows this flow in between its jumps:
\begin{equation}\label{e:detflow}
X_t = (N_{T_k}, \phi(t,X_{T_k}))=
 (N_{T_k}, R_{T_k} - (t-T_k)), T_k < t <  T_{k+1}.
\end{equation}
For $A \subset {\mathbb R}^2$, let $\partial A$ denote its boundary
in the Euclidian topology.
The exit boundary of the process is
\[
\Gamma^* \doteq \cup_{k=0}^\infty \partial E_k = \partial B({\bm 0},\delta) \cup \left( \cup_{k=1}^\infty \{ (k,0), k > 0 \}\right).
\]
For $x  = (k,r) \in E$, define (following \cite[page 57]{davis1993markov})
$t_*(x) \doteq  \inf \{t> 0, \phi_k(t,r) \in \partial E_k\}$
where we use the convention that the infimum of the empty set
is empty; $t_*(x)$ is the time when $X$ reaches $\Gamma^*$ if it doesn't
doesn't jump until this happens.
 By definition (\eqref{e:detdynamics} and \eqref{e:detflow})
$X$ moves with unit speed toward the $k$-axis 
on each $E_k$, $k > 0$, therefore,
\[
t_*(x) = r, x =(k,r) \in E, k > 0.
\]
For $k=0$, the process remains constant ${\bm 0}$
 until an arrival occurs, which implies
$ t_*({\bm 0})= \infty.$
Figure \ref{f:dynamics} shows an example sample path of $X$; 
the horizontal axis is the
$k$-axis, showing the number of customers in the system and
the vertical axis is the $r$-axis,
showing the remaining service time of the current customer in service.
The dynamics
\eqref{e:detflow} means that $X$ travels with unit speed toward the $k$-axis 
in between its jumps.
Two types of jumps are possible:
either an arrival, which are jumps to the right or a service completion, 
which are jumps to the left occurring when $X$ hits the $k$-axis.

The jump dynamics are specified by the rate function $
{\boldsymbol \lambda}:E\rightarrow
{\mathbb R}_+$ and the transition measure $Q$.
For the $rM/G/1$
system the jump rate function will be
\[
{\boldsymbol \lambda}(k,r) \doteq \begin{cases} \lambda_0 f(r), & k >0, \\
					\lambda_0, & k=0.
\end{cases}
\]
The transition measure $Q(\cdot,x)$, $x \in E \cup \Gamma^*$ for
the $rM/G/1$ system will be as follows:
$Q(\cdot,x)$ is the Dirac measure on $(k+1,r)$ for $x=(k,r)$,
$k > 0$ and $r > 0$ (represents an arrival to the busy system).
 For $(k,0) \in \Gamma^*$, $k > 0$,
$Q(\cdot,x)$ is the measure $g(r) dr$ on $E_{k-1}$ (represents
the completion of a service and the start of another, this is exactly when the sample
path $X$ hits the $k$-axis in Figure \ref{f:dynamics});
$Q(\cdot, {\bm 0})$ is the measure $g(r)dr$ on $E_1$ (represents
an arrival to the empty system).

\begin{figure}[h]
\begin{center}
\scalebox{0.8}{
\centerline{\input{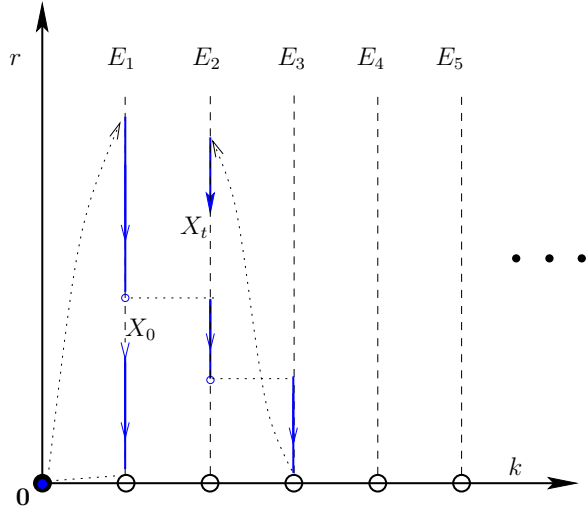}}}
\end{center}
\caption{\hspace{0.25cm}The state space and a sample path of $X$\label{f:dynamicstex}}
\end{figure}

\subsection{Generator of $X$}\label{ss:generator}
Let ${\mathscr E}$ denote  the $\sigma$-algebra of
Borel-measurable subsets of $E$.
Let $\{T_n, n=1,2,3,..\}$ denote the jump times of $X$. 
For 
$h:E \times {\mathbb R}_+ \times \Omega \mapsto {\mathbb R}$, $h$ measurable,
one writes $h \in L_1(X)$ if
\[{\mathbb E}\left[ \sum_{i=1}^\infty \left| h(X_{T_i}, T_i,\omega)\right| 
\right]
 < \infty
\]
and $h \in L_1^{\text{loc}}(X)$ if $h \indic_{\{ t< \sigma_n\}} \in L_1(X)$
for a sequence of stopping times $\sigma_n \nearrow \infty.$ 
The characterization of the generator of $X$ given in the
next paragraph uses these definitions.

The generator of any PDP process is derived explicitly in
\cite[Theorem (26.14), page 69]{davis1993markov}; for the 
$rM/G/1$ process $X$, it is given by the following operator:
\[
{\mathfrak A}h (x) =\begin{cases}
-\frac{d }{dr} h(x) +\lambda_0 f(r)
\left( h(k+1,r)  - h(k,r)\right), x =(k,r), k,r > 0,\\
+\lambda_0 g(r) (h(1,r) - h({\bm 0})), x = {\bm 0},
\end{cases}
\]
where 
$h \in {\mathscr D}({\mathfrak A})$, the domain of ${\mathfrak A}$,
consisting of measurable functions $h$ on $E \cup \Gamma^*$ satisfying:
\begin{enumerate}
\item for each $k > 0$, $h(k,\cdot)$ is absolutely continuous on 
${\mathbb R}_+$,
\item  $h(k,0) = \int h(k,r) g(r) dr, k > 0$ and 
\item 
${\mathfrak B} h \in L_1^{\text{loc}}(X)$ where
${\mathfrak B} h$ is the process $t \mapsto ( h(X_0) - h(X_t-)).$
\end{enumerate}

\subsection{Embedded random walk at service completion times}\label{ss:completion}
Let $S_k$ denote\footnote{In \cite{davis1993markov}, $S_k$ denotes the
inter-jump times of the PDP, here they denote the successive times when the process
hits the boundary of its state space.}
the sequence of service completion times
\[
S_1 \doteq \inf\{ t: X_{t-} \in \Gamma^*  \}, S_n \doteq \inf
 \{t > S_{n-1}, X_{t-} \in \Gamma^* \}, n > 1,
\]
and define the process
$({\mathcal N}_n, {\mathcal R}_n) \doteq X_{S_n}$,
the state of the system right after service completions. 
Let
\[
{\boldsymbol p}(k,\lambda) =  \frac{e^{-\lambda} \lambda^k}{k!}, k=0,1,2,3,...
\] denote the Poisson distribution with rate
$\lambda$
and define
\[
F(r) \doteq   \int_0^r f(r-u)du =  \int_0^r f(u)du.
\]
\begin{proposition}
The process $\{({\mathcal N}_n, {\mathcal R}_n),n=1,2,3,...\}$ 
is a Markov chain with transition probabilities
\begin{align}\label{e:CQR1}
P( {\mathcal N}_{n+1} - {\mathcal N}_{n} =j | {\mathcal N}_n,{\mathcal R}_n)
&= 
\begin{cases}
{\boldsymbol p}(j+1,F({\mathcal R}_n)),& {\mathcal N}_n > 0,\\
 {\boldsymbol p}(j, F({\mathcal R}_n)),& {\mathcal N}_n = 0
\end{cases}\\
P({\mathcal R}_{n+1} \in A | {\mathcal N}_{n+1}) &= 
\begin{cases}
\delta_{0}(A), & {\mathcal N}_{n+1} = 0, \\
\int_A g(r) dr, & {\mathcal N}_{n+1} > 0,
\end{cases}\label{e:distofR}
\end{align}
where $\delta_0$ denotes the Dirac measure on $0$.
\end{proposition}
\begin{proof}
The definition of the process $X$ (or its strong Markov property) 
implies that $({\mathcal N},{\mathcal R})$ is a
Markov chain.
The jump distribution $Q$ determines where $X$ jumps after it
hits $\Gamma^*$: $X$ jumps to $(N_{S_n-}-1, \sigma_n)$
where $\sigma_n$ has density $g$,  if $N_{S_n-} > 1$ or it jumps
to ${\bm 0} = (0,0)$ if $N_{S_n-} = 0.$
This gives \eqref{e:distofR}. 
To compute the conditional density of ${\mathcal N}_{n+1}$ given
$({\mathcal N}_n,{\mathcal R}_n)$ it suffices to compute that
of
\begin{equation}\label{e:decompose0}
{\mathcal N}_{n+1} -{\mathcal N}_n = 
( N_{S_{n+1}} - N_{S_{n+1}-})+ (N_{S_{n+1} - } - N_{S_n}).
\end{equation}
By the strong Markov property of $X$,
the conditional distribution of 
${N}_{S_{n+1}-} - {N}_{S_n}$ given $({\mathcal N}_n,{\mathcal R}_n)$
is the same as that of
$N_{S_1-} - N_0$ 
given $(N_0, R_0)$. The two cases $(N_0,R_0) = {\bm 0}$ and
$N_0, R_0 > 0$ are treated separately.
Let us start with the latter:
conditioned on $(X_0 = (N_0,R_0) = x= (k,r))$, $k > 0$, the dynamics of $X$
imply the following:  $S_1 = t_*(x) = r$,  and $R_t = r-t$ for
$t \in [0,r).$ In the same time interval the $N$ process is Poisson
with time dependent rate $\lambda_0 f(r-t).$ 
Therefore, conditioned on $X_0 = (k,r)$, $k > 0$,
$N_{S_1 - } - N_0$
has Poisson distribution with rate  $F(r)$.
Furthermore, for $k>0$, one has $N_{S_1-} > 0$ and therefore, once again
by the definition of the jump dynamics of $X$, $N_{S_1} - N_{S_1-} = -1$
(i.e., the customer whose service has just finished leaves the system).
These and \eqref{e:decompose0} imply
\begin{equation}\label{e:CQR1p}
P( {\mathcal N}_{n+1} - {\mathcal N}_{n} =j | {\mathcal N}_n,{\mathcal R}_n)
= {\boldsymbol p}(j+1,F({\mathcal R}_n)), {\mathcal N}_n > 0.
\end{equation}
The argument for the case $X_0= {\bm 0}$ is parallel and gives
\begin{equation}\label{e:CQR2}
P( {\mathcal N}_{n+1} - {\mathcal N}_{n} =j | 
({\mathcal N}_n,{\mathcal R}_n) = {\bm 0})
= {\boldsymbol p}(j,F(R_{\tau'_n}))
\end{equation}
where $\tau'_n$ is the first jump time of $R$ after $S_n$. 
For
$({\mathcal N}_n,{\mathcal R}_n) = {\bm 0}$, $\tau'_n$ will be a jump
from state ${\bm 0}$ (i.e., an arrival to the empty system) 
and by $X$'s definition 
$R_{\tau'_n}$'s density, given
the whole history of $({\mathcal N},{\mathcal R})$,  will again be $g$.
This, \eqref{e:CQR1p} and \eqref{e:CQR2} imply \eqref{e:CQR1}.
\end{proof}
The process 
${\mathcal N}$ itself is a Markov chain:
\begin{proposition} \label{p:MG1reduce}
${\mathcal N}$ is a Markov chain with transition matrix 
\begin{equation}\label{d:M}
M= \left(
\begin{matrix}
p(0) & p(1)    & p(2) & p(3) & p(4) & \cdots \\ 
p(0) & p(1)    & p(2) & p(3) & p(4) & \cdots \\ 
0      & p(0) & p(1) & p(2) & p(3)& \cdots \\
0      & 0      & \ddots & \ddots & \ddots & \ddots
\end{matrix}
\right),
\end{equation}
where
\begin{equation}\label{e:defpj}
p(j) \doteq \int_0^\infty {\boldsymbol p}(j, \lambda_0 F(r)) g(r) dr.
\end{equation}
\end{proposition}
\begin{proof}
The conditional distributions \eqref{e:CQR1}, \eqref{e:distofR} 
and
the Markov property of
the process $\{({\mathcal N}_n, {\mathcal R}_n)\}$,
imply that
$\{ {\mathcal N}_n,n=0,1,2,3,..\}$ is a Markov chain;
the distribution of
its increments 
$\Delta {\mathcal N}_n = {\mathcal N}_{n+1} - {\mathcal N}_n$
are
\begin{equation}\label{e:conddistinc}
P( \Delta {\mathcal N}_{n} = j | {\mathcal N}_n)
= { \mathbb E}[{\mathbb E}
[( \indic_{{\mathcal N}_{n} = j} | {\mathcal N}_n, R_n)  ]
{\mathcal N}_n]
= \begin{cases} p(j+1), &{\mathcal N}_n > 0, \\
	 p(j), &{\mathcal N}_n =0,
\end{cases}
\end{equation}
This implies that $M$ of \eqref{d:M} is the transition
matrix of ${\mathcal N}$.
\end{proof}
We now note the first connection between $rM/G/1$ and $M/G/1$ systems. 
That $\{\sigma_i\}$ is an iid sequence implies the same for
$\bar{\sigma}_i\doteq F(\sigma_i).$
Then one can write \eqref{e:defpj} as
\[
p(j) = {\mathbb E}[{\boldsymbol p}(j,\lambda_0 \bar{\sigma}_1)],
j=0,1,2,3,,...
\]
and, by \cite[Proposition 3.3.2, page 57]{meyn2012markov}, these are
exactly the transition probabilities of the embedded random
walk (at service completion times) of an $M/G/1$ system with
constant rate $\lambda_0$ and iid service time sequence 
$\{\bar{\sigma}_i, i=1,2,3,..\}$:
\begin{corollary}\label{c:MG1eq1}
The dynamics
at service completion times
of the $rM/G/1$ system with arrival rate $\lambda_0 f(\cdot)$
and iid sequence of service times $\{\sigma_i,i=1,2,3,...\}$
is identical to the dynamics at
service completion times
of an $M/G/1$ system with constant arrival rate $\lambda_0$ and
iid sequence of service times $\{\bar{\sigma}_i = F(\sigma_i), i=1,2,3,...\}.$
\end{corollary}

The next section computes the stationary distribution of ${\mathcal N}$ under
a natural stability assumption;  before
we move on, let us make the following observation:
\begin{remark}
{\em
Let $E_t$ denote the
elapsed service time since the beginning of current service.
If we replace
the arrival rate from $\lambda_0 f (R_t)$ with $\lambda_0 f(E_t)$,
conditioned on ${\mathcal R}_n = r$, the number of arrivals  between
the $n^{th}$ service completion and $(n+1)^{st}$ completion will be
a Poisson random variable with rate 
$\lambda_0 \int_0^r f(u)du$, i.e., the same as that of the $rM/G/1$
system; therefore,
the transition matrix $M$ of the embedded walk ${\mathcal N}$
remains unchanged if we replace the arrival rate $\lambda_0 f(R_t)$
with $\lambda_0 f(E_t).$
This implies that all of our
computations concerning ${\mathcal N}$
above and in Section \ref{s:servcomp} below 
remain unchanged if the arrival rate process is changed from
$\lambda_0 f (R_t)$ to $\lambda_0 f(E_t)$.
}
\end{remark}

\section{Stationary distribution at service completions or arrival times}\label{s:servcomp}

A measure $q$ is the stationary
measure of ${\mathcal N}$ if and only if it satisfies
\begin{equation}\label{e:balanceq}
q = qM.
\end{equation}
We have seen in Corollary \ref{c:MG1eq1}
that the dynamics of the $rM/G/1$ system at service
completion times is identical to that of the $M/G/1$ system
with constant arrival $\lambda_0$ and service times 
\{$\bar{\sigma}_i = F(\sigma_i)$\}, therefore
\eqref{e:balanceq} is also the balance equation  of this $M/G/1$ system
at its service completion times. The well known solution of this system
is (see, e.g., \cite[page 238]{meyn2012markov} 
or \cite[page 281]{asmussen2008applied})
\begin{equation}\label{e:explicitq}
q(j) = q(0) \bar{p}(j-1) + \sum_{i=1}^{j-1} q(i) \bar{p}(j-i), j=1,2,3,...
\end{equation}
where
$\bar{p}(j) \doteq  \sum_{i=j+1}^\infty p(j).$
In particular, a (possibly degenerate) invariant distribution
always exists and is uniquely defined as soon as $q(0)$ is fixed.
By definition $q$ is nondegenarate if and only if
$\sum_{i=1}^\infty q(i) < \infty$, i.e., if $q$ is a finite measure on
${\mathbb N}.$ 
\cite[Proposition 10.3.1, page 239]{meyn2012markov} gives precisely the condition
for this to hold:
\begin{proposition}\label{p:invariantmeasure}
$q$ of \eqref{e:explicitq} defines a finite measure if and only if
$-1 + \sum_{n}n p(n)< 0$, i.e., if
\begin{equation}\label{as:stable}
 \rho \doteq \lambda_0 \bar{\nu} = 
\lambda_0 {\mathbb E}[\bar{\sigma}_i]=
\lambda_0 {\mathbb E}[F(\sigma_i)]=
\lambda_0 \int_0^\infty F(r)g(r)dr< 1.
\end{equation}
\end{proposition}
Then, under the stability condition \eqref{as:stable},
$q(0)$ can be chosen so that $\sum_{i=0}^\infty q(i) = 1$.
and, with this choice, $q$ will be the unique stationary
measure of the process ${\mathcal N}.$ 
To determine the value of $q(0)$ for which $q$ is a proper probability 
measure, following \cite[page 239]{meyn2012markov}, one sums
both sides of \eqref{e:explicitq} to get
\[
\sum_{j=1}^\infty q(j) = q(0) \frac{\rho}{1-\rho};
\]
then, for $\sum_{i=0}^\infty q(i) = 1$ we must have
\[
q(0) = 1 -\rho.
\]
In the rest of this article, 
we will take $q(0) = 1-\rho$
whenever the stability assumption 
\eqref{as:stable} is made.
Under these assumptions
$q(0)$ is the stationary limit probability
of an empty $rM/G/1$ queue right after service completions:
\begin{proposition}
The distribution of ${\mathcal N}_n$ converges in total variation norm 
to $q$.
In particular,
\[
\lim_{n\rightarrow \infty} P({\mathcal N}_n = 0) =  q(0)= 1-\rho = 
1- \lambda {\mathbb E}[F(\sigma_1)]=1 -\lambda_0 \int_0^\infty F(r)g(r)dr.
\]
\end{proposition}
\begin{proof}
That $q$ is the stationary distribution of ${\mathcal N}$ follows from
\eqref{e:balanceq}. ${\mathcal N}$ is strongly aperiodic; by
\cite[Proposition 10.3.1]{meyn2012markov} it is positive when
\eqref{as:stable} holds. The convergence in total variation
norm follows from these and \cite[Theorem 13.3.1]{meyn2012markov}.
\end{proof}

By Corollary \ref{c:MG1eq1}, all results/computations for
the $M/G/1$ queue at service completion times hold for 
the $rM/G/1$ queue. For example, the expected queue length
at service completion times, is given by the Pollaczek-Khinchine formula
\cite[Equation (5.3), page 281]{asmussen2008applied}
\begin{equation}\label{e:pk1}
{\mathbb E}_q[{\mathcal N}_1] = \sum_{k=1}^\infty k q(k)=
\rho + \frac{ \lambda_0^2{\mathbb E}[F(\sigma_1)^2]}{2(1-\rho)},
\end{equation}
where the subscript $q$ of ${\mathbb E}$ denotes that the Markov chain
${\mathcal N}$ is run in its stationary distribution,
and the moment generating function of the stationary distribution
is \cite[(5.8), page 283]{asmussen2008applied}:
\begin{equation}\label{e:mgf}
{\mathbb E}_q[e^{s{\mathcal N_1}}]
= 
\sum_{k=1}^\infty e^{s k} q(k) = \frac{(1-\rho)(1-s)\psi_p(s)}{\psi_p(s) -s} 
\end{equation}
where $\psi_p$ is the moment generating function of the
increments of ${\mathcal N}$:
\[
\psi_p(s) ={\mathbb E}\left[e^{\lambda_0(1-s) \bar{\sigma}_1} \right]
 ={\mathbb E}\left[e^{\lambda_0(1-s) F(\sigma_1)} \right].
\]
\paragraph{Stationary distribution at arrival times}
Let $S^A_n$ be the sequence of arrival times to the system.
Then $({\mathcal N}^A_n, {\mathcal R}^A_n) = (N_{S^A_n}-1, R_{S^A_n})$
is the embedded Markov chain of $X$ representing the state of the
system just before arrivals. The fact that the queueing process $X$
changes in increments of $1$ and $-1$ exactly at arrival and service completion
times imply that, under the stability assumption \eqref{as:stable},
the process ${\mathcal N}^A$ will also have stationary 
distribution $q$, the stationary distribution
of ${\mathcal N}.$ For details of similar arguments we refer the reader to
\cite[Theorem 4.3, page 278]{asmussen2008applied} or
\cite[Section 5.3]{kleinrock75}.
\section{Stationary distribution in continuous time}\label{s:stationarycont}
One of the key properties of $M/G/1$ systems is that their stationary
queue length distribution at service completion times is equal to the
same distribution under their continuous time stationary measure. We will
see in Corollary \ref{c:emptysystemprob} below that the $rM/G/1$ system
does not possess this property, hence the continuous time stationary
measure and related performance measures (such as the average waiting time)
for the $rM/G/1$ queue have to be computed separately. 
This is the goal of the present section.  The following verification argument 
will give us the stationary distribution of $X$:
\begin{enumerate}
\item Derive the balance equation for the stationary distribution,
\item Solve the balance equation,
\item Invoke \cite[Proposition (34.7), page 113]{davis1993markov} to show
that the solution is indeed the stationary measure of $X$ (see
Proposition \ref{p:proofstationary} below).
\end{enumerate}

For a measure $\mu$ on $E$ and $k \in \{1,2,3,...\}$,
we say that $\mu$ has density $m$ on $E_k$,
if $\mu(A \cap E_k) = \int_0^\infty 1_A((k,r)) m(r) dr$, for any
measurable $A \subset E.$ Define
\[
{\mathscr M} \doteq \{\mu: \text{ is a measure on $E$ having density
$m(k,\cdot)$ on }E_k,~k=1,2,3,...\}.
\]
The balance equation for the stationary distribution is
\begin{equation}\label{e:balance}
{\mathfrak A}^*(\mu)(x) = 0, x \in E.
\end{equation}
where
${\mathfrak A}^*$ 
is the conjugate 
operator (acting on measures $\mu \in {\mathscr M}$)
of the generator operator ${\mathfrak A}$:
\begin{equation}\label{e:defA}
{\mathfrak A}^*(\mu)(x) = 
\begin{cases}
\frac{d}{dr} m(k,r)  + \lambda_0 f(r) (m(k-1,r) - m(k,r)) + m(k+1,0)g(r), k > 1, r > 0\\
\frac{d}{dr} m(1,r) + \lambda_0 \mu({\bm 0}) g(r) + m(2,0) g(r) -  \lambda_0 f(r) m(1,r), r > 0\\
m(1,0) - \mu({\bm 0}) \lambda_0.
\end{cases}
\end{equation}
The goal of this section is to show that (up to scaling) there is a
unique solution $\mu^*$ to the balance equation \eqref{e:balance}
and this solution is the
stationary measure of the continuous time $rM/G/1$ process $X$.

Keep $\mu^*({\bm 0})  > 0$ as a free parameter to be
fixed below.
The third line of \eqref{e:defA} gives
\begin{align}
{\mathfrak A}^*(\mu^*)({\bm 0}) = 
\mu^*({\bm 0})
\lambda_0 -m^*(1,0) &= 0 \notag \\
m^*(1,0) &= \label{e:barm10}
\mu^*({\bm 0}) \lambda_0.
\end{align}
The last equality and the second line of \eqref{e:defA} imply
that \eqref{e:balance} reduces to
the following equation for $m(1,\cdot)$:
\begin{equation}\label{e:odel2}
\frac{d}{dr} m(1,r) + g(r)( {m}^*(1,0) + m^*(2,0))
 -  \lambda_0 f(r) m(1,r)   = 0 , r > 0.
\end{equation}
The classical linear ODE theory 
implies that the unique 
solution of \eqref{e:odel2} vanishing at $\infty$ is
\begin{equation}\label{e:barm1s}
m^*(1,r) =  
\left( m^*(1,0) + m^*(2,0)\right)
\int_{r}^{\infty} g(u)e^{(F(r)-F(u))\lambda_0} du,
\end{equation}
Substituting $r=0$ gives the following formula for $m^*(2,0)$:
\begin{align}
m^*(1,0) &= 
\left( m^*(1,0) + m^*(2,0)\right) p(0) \label{e:barm20i}
\intertext{ or }
m^*(2,0) &\doteq 
\frac{ 1-p(0)}{p(0)}
m^*(1,0)  > 0,
\label{e:barm20}
\end{align}
where, 
\[
p(0) =  \int_0^\infty g(r)e^{-F(r)\lambda_0} dr,
\]
 is the $0$ increment probability of the embedded
Markov chain ${\mathcal N}$, given in \eqref{e:defpj}.
That $m^*(2,0) > 0$ implies $m^*(1,\cdot) >  0.$
Next derive a second expression for $m^*(2,0)$
by integrating both sides of \eqref{e:barm1s} over $[0,\infty)$:
\begin{align*}
\int_0^\infty m^*(1,r)f(r)dr &=  
\left( m^*(1,0) + m^*(2,0)\right)
\int_0^\infty f(r)\int_r^{\infty} g(u)e^{(F(r)-F(u))\lambda_0} du dr\\
&=
\left( m^*(1,0) + m^*(2,0)\right)
\frac{1}{\lambda_0} ( 1 - p(0)),
\end{align*}
where we have used Fubini's theorem, $m(1,\cdot) > 0$
 and  the change of variable $s= F(r).$
The definition \eqref{e:barm20} of $m^*(2,0)$ implies
$m^*(1,0) = \frac{p(0)}{1-p(0)} m^*(2,0)$; substituting this
in the last line above gives
\begin{equation}\label{e:integralrepbm20}
m^*(2,0) = 
\lambda_0 \int_0^\infty m^*(1,r) f(r) dr
\end{equation}
Formulas \eqref{e:barm10}, \eqref{e:barm1s} and \eqref{e:barm20}
uniquely determine $m^*(1,\cdot)$ and $m^*(2,0)$
given $\mu^*({\bm 0}).$
For $k>1$,
\eqref{e:balance}
uses the first line of \eqref{e:defA}:
\begin{equation}\label{e:odepart3}
\frac{d}{dr} m(k,s)  
+ \lambda_0 f(r) (m(k-1,s) - m(k,s)) + m(k+1,0)g(r) =0, r>0.
\end{equation}
The unique solution of this linear equation for $k=2$ decaying
at $\infty$ is
\begin{equation}\label{e:barmk2}
m^*(2,r)=
m^*(3,0)\int_{r}^{\infty} g(u) e^{(F(r)-F(u))\lambda_0}du
+\lambda_0 \int_{r}^{\infty}e^{(F(r)-F(u))\lambda_0} m^*(1,u)f(u)du,
\end{equation}
where $m^*(3,0)$ is yet to be determined. To determine it
 set $r=0$ in the above display to get
\begin{equation}\label{e:barm30}
m^*(3,0) =
\frac{1}{p(0)}\left(
m^*(2,0) - \lambda_0\int_0^\infty e^{-F(u)\lambda_0} m^*(1,u)f(u)du\right);
\end{equation}
With this, $m^*(2,\cdot)$ and $m^*(3,0)$ are determined uniquely,
given $\mu({\bf 0})$. 
\eqref{e:integralrepbm20} and the definition of 
$m^*(3,0)$ imply $m^*(3,0) > 0$, which in
its turn implies $m^*(2,\cdot) > 0$.

Letting $r\rightarrow \infty$ in \eqref{e:barmk2} gives
$\lim_{r\rightarrow \infty}m^*(2,r) =0.$ This and the integration of 
\eqref{e:odepart3} on $[0,\infty)$ gives
\[
-m^*(2,0) + \lambda_0 \int_0^\infty m^*(1,r)f(r) dr
-\lambda_0 \int_0^\infty m^*(2,r) f(r) dr + m^*(3,0)  = 0.
\]
This and
\eqref{e:integralrepbm20} now imply a similar equation
for $m^*(3,0)$:
\begin{equation}\label{e:integralrepbm30}
m^*(3,0)  = \lambda_0 \int_0^\infty m^*(2,r)f(r) dr.
\end{equation}

For $k > 2$, one solves \eqref{e:odepart3}  inductively, using
$k=2$ as the base case to get the following sequence of 
unique positive solutions of \eqref{e:odepart3}
vanishing at $\infty$:
\begin{equation}\label{e:barmk2g0}
m^*(k+1,0) \doteq
\frac{1}{p(0)}\left(
 m^*(k,0) - \lambda_0\int_0^\infty e^{-F(u)\lambda_0} 
m^*(k-1,u)f(u)du\right)
\end{equation}
\begin{align}\label{e:barmk2g}
m^*(k,r) &\doteq
m^*(k+1,0)\int_{r}^{\infty} g(u) e^{(F(r)-F(u))\lambda_0}du\\
&~~~~~~~~~~+\lambda_0 \int_{r}^{\infty}e^{(F(r)-F(u))\lambda_0} m^*(k-1,u)f(u)du, 
\notag
\end{align}
$r>0$, and the solution satisfies
\begin{equation}\label{e:integralrepbmk0}
m^*(k+1,0)  = \lambda_0 \int_0^\infty f(r)m^*(k,r) dr.
\end{equation}
The last formulas are the extension of
\eqref{e:barmk2} and \eqref{e:barm30} 
 to $k>2$.
Let us note the foregoing computations as a proposition:
\begin{proposition}\label{p:muexistsunique}
Given $\mu^*({\bf 0}) > 0$, 
the balance equation \eqref{e:balance} has a unique positive solution
$\mu^*$
given by \eqref{e:barm10},  \eqref{e:barm1s}, \eqref{e:barm20}
for $k=1$ and \eqref{e:barmk2}, recursively, for $k\ge 2.$ The solution satisfies
\begin{equation}\label{e:integralrepbmk0}
m^*(k+1,0) = \lambda_0 \int_0^\infty m^*(k,r)f(r) dr
\end{equation}
for $k \ge 1$.
\end{proposition}

The next proposition links the quantities $m^*(k,0)$
to the stationary distribution of the embedded chain ${\mathcal N}$:
\begin{proposition}\label{p:finitenessofmubar}
Let
 $\mu^*=(\mu^*({\bm 0}), m^*(k,\cdot), k=1,2,3,...)$ 
be the unique solution (up to the choice of $\mu^*({\bm 0}) > 0$)
of the balance equation \eqref{e:balance} derived in Proposition
\ref{p:muexistsunique} above. Then the measure 
\[
 m^*  \doteq (m^*(1,0), m^*(2,0), m^*(3,0),\cdots )
\]
 on ${\mathbb N}_+$
is $M$-invariant, i.e.,
\begin{equation}\label{e:eigenvector}
 m^* M = m^*
\end{equation}
and
\begin{equation}\label{e:equivalence}
m^* = c q
\end{equation}
for some $c > 0$ where $q$ is the stationary measure given in 
\eqref{e:explicitq}.
In particular,
\begin{equation}\label{e:summfinite}
\sum_{k=1}^\infty m^*(k,0) < \infty
\end{equation}
if the stability assumption \eqref{as:stable} holds.
\end{proposition}
\begin{proof}
By definition \eqref{d:M} of the matrix $M$ 
\eqref{e:eigenvector} is the following sequence
of equations:
\begin{equation}\label{e:toprove0}
p(n)(m^*(1,0) + m^*(2,0)) + 
\sum_{k=2}^{n+1} m^*(k+1,0) p(n+1-k) = m^*(n+1,0),
\end{equation}
$n=0,1,2,3,...$
For $n=0$, \eqref{e:toprove0} reduces to \eqref{e:barm20i}, which holds by
definition.
To prove \eqref{e:toprove0} for $n>0$, 
multiply both sides of \eqref{e:odel2} by 
$e^{-F(r) \lambda_0 } \frac{(F(r)\lambda_0)^{n}}{n!}$, 
integrate from $0$ to $\infty$ to get
\begin{align}
0&=
\int_0^\infty \frac{d}{dr} m^*(1,r)
e^{- F(r)\lambda_0 } \frac{(F(r)\lambda_0)^{n}}{n!}dr
 + ( m^*(1,0)+ m^*(2,0))p(n) \notag \\
&~~~~~~~~~~~~-  \lambda_0 \int_0^\infty m^*(1,r) \notag
e^{-F(r) \lambda_0 } \frac{(F(r)\lambda_0)^{n}}{n!}f(r)dr.
\intertext{Integration by parts on the first integral gives:}
0&=(m^*(1,0) + m^*(2,0))p(n) 
-\lambda_0 \int_0^\infty  \label{e:casen}
e^{-F(r) \lambda_0 } \frac{(F(r)\lambda_0)^{n-1}}{(n-1)!}m^*(1,r)f(r)dr.
\end{align}
For $k =2,...,n+1$, multiply both sides of \eqref{e:odepart3} by
$e^{-F(r) \lambda_0 } \frac{(F(r)\lambda_0)^{n+1-k}}{(n+1-k)!}$
and integrate by parts the first term to get
\begin{align}\label{e:restofthek}
m^*(k+1,0) p(n+1-k) +&\lambda_0 \int_0^\infty 
e^{-F(r) \lambda_0 } \frac{(F(r)\lambda_0)^{n+1-k}}{(n+1-k)!}m^*(k-1,r)f(r)dr \\
&~~~~~~-\lambda_0 \int_0^\infty  \notag
e^{-F(r) \lambda_0 } \frac{(F(r)\lambda_0)^{n-k}}{(n+1-k)!}m^*(k,r)f(r)dr = 0.
\end{align}
Summing  the last display over $k$, adding to the result \eqref{e:casen}
and finally noting \eqref{e:barmk2g0}
give \eqref{e:toprove0} for $n > 0.$
The Markov chain ${\mathcal N}$ is a constrained random walk on ${\mathbb Z}_+$
with iid increments
and hence is obviously irreducible and will therefore have (up to scaling)
a unique stationary distribution; \eqref{e:equivalence} follows from
this. 
\eqref{e:summfinite} follows from \eqref{e:equivalence}
and Proposition \ref{p:invariantmeasure}.
\end{proof}

Define
\[
S(r) \doteq  \sum_{k=1}^\infty m^*(k,r),
\]
whose finiteness under the stability assumption follows
from 
\eqref{e:integralrepbmk0} and
the previous proposition (see \eqref{e:summfinite});
\eqref{e:integralrepbmk0} also implies
\begin{equation}\label{e:S00}
\lambda_0 \int_0^\infty S(r) f(r) dr = S(0) - m^*(1,0) = S(0) -\lambda_0\mu^*({\bm 0}).
\end{equation}
Remember that $\mu^*({\bm 0})$ is still a free parameter.
The next proposition computes $\int_0^\infty S(r)dr$ and $S(0)$
in terms of $\mu^*({\bm 0})$ and in terms of the system parameters.
\begin{proposition}\label{p:mu1starE}
Suppose the stability assumption \eqref{as:stable} holds.
Then
\begin{equation}\label{e:formulaforS0}
S(0) = \lambda_0 \frac{\mu^*({\bm 0})}{1-\rho}.
\end{equation}
and
\begin{equation}\label{e:formulaforintS}
\int_0^\infty S(r)dr= S(0) \nu.
\end{equation}
\end{proposition}
\begin{proof}
Summing the terms of the balance equation gives 
$S'(r) = -S(0)g(r)$, therefore,
\begin{equation}\label{e:Sr}
S(r) = S(0)G(r),
\end{equation}
where
\begin{equation}\label{d:Gr}
G(r) \doteq P(\sigma_1 > r) = \int_r^\infty g(u)du.
\end{equation}
Integrating
both sides of \eqref{e:Sr} over $[0,\infty)$ gives \eqref{e:formulaforintS}.
Next multiply both sides by $f(r)$ and integrate over $[0,\infty]$:
\[
\int_0^\infty S(r) f(r) dr = S(0) \int_0^\infty f(r) G(r)dr = S(0) \bar{\nu},
\]
where we have integrated by parts the middle integral.
The last display and \eqref{e:S00} imply
\begin{align*}
\lambda_0 S(0) \bar{\nu} &= S(0) -\lambda_0\mu^*({\bm 0})\\
S(0)  &= \lambda_0 \mu^*({\bm 0}) \frac{1}{1-\rho},
\end{align*}
which proves \eqref{e:formulaforS0}. 
\end{proof}
Let us fix the value for $\mu^*({\bm 0})$ to
\begin{equation}\label{e:setmu0}
\mu^*({\bm 0}) = 1-\rho = q(0);
\end{equation}
we will assume \eqref{e:setmu0} whenever the stability assumption \eqref{as:stable} is made.
This implies
by \eqref{e:formulaforS0} and \eqref{e:Sr}:
\begin{equation}\label{e:valueofS0}
S(0) = \lambda_0, S(r) =\lambda_0 G(r).
\end{equation}
A second implication is given in the next lemma.
\begin{lemma}
Let  $\mu^*({\bm 0})$ be fixed as in \eqref{e:setmu0}, i.e.,
we take $\mu^*({\bm 0}) = q(0)$. Then
\begin{equation}\label{e:exactmultiplier}
m^* = \lambda_0 q
\end{equation}
where $m^*=(m^*(1,0),m^*(2,0),...)$ is as in Proposition \ref{p:finitenessofmubar}. In particular,
\begin{equation}\label{e:expectedlength1}
\sum_{k=1}^\infty k m^*(k+1,0) = \lambda_0 \sum_{k=1}^\infty kq(k) = \lambda_0 {\mathbb E}_q[{\mathcal N}_1].
\end{equation}
\end{lemma}
\begin{proof}
We know by \eqref{e:equivalence} that $m^* = cq$ for some $c > 0.$ 
Because
$q$ is a probability measure, this implies, 
$c = \sum_{k=1}^\infty m^*(k,0)=S(0)$, which equals $\lambda_0$,
by \eqref{e:valueofS0}. This proves \eqref{e:exactmultiplier};
\eqref{e:expectedlength1} follows from \eqref{e:exactmultiplier}.
\end{proof}

With $\mu^*({\bm 0})$ fixed as in \eqref{e:setmu0}, the measure
$\mu^*$ is determined uniquely via Proposition \ref{p:muexistsunique}.
Note
\begin{equation}\label{e:muE}
\mu^*(E) = \mu^*({\bm 0}) + \int_0^\infty S(r)dr
= 1- \rho + \lambda_0 \nu.
\end{equation}
where we have used \eqref{e:formulaforintS},
 \eqref{e:formulaforS0} and \eqref{e:setmu0}.
Thus, in general, with $\mu^*({\bm 0})$ fixed as in \eqref{e:setmu0},
$\mu^*(E) \neq 1$- to get a proper probability measure, renormalize 
$\mu^*$:
\[
\mu^*_1 \doteq \mu^*/\mu^*(E).
\]
 Proposition \ref{p:proofstationary} below proves that
$\mu^*_1$ is the unique stationary measure of the $rM/G/1$ process
$X$ under the stability assumption \eqref{as:stable}.
The proof will require a
subclass of functions in ${\mathscr D}({\mathfrak A})$
that can separate measures in ${\mathscr M}.$  The following
lemma identifies such a class.
\begin{lemma}\label{l:separating}
\[
{\mathscr S} \doteq \{h \in {\mathscr D}({\mathfrak A}), 
\sup_{x \in E}|h'(x)| < \infty, \sup_{x \in E}|h(x)| <  \infty \}
\]
is a separating class of functions for measures in ${\mathscr M}.$
\end{lemma}
\begin{proof}
For $\mu_1, \mu_2 \in {\mathscr M}$, $\mu_1 = \mu_2$ if and only if
\[
\int_a^b m_{1,k}(r)dr = \int_a^b m_{2,k}(r)dr
\]
for all $0 < a < b < \infty$ and $k > 0$ ($m_{1,k}$ and $m_{2,k}$
are densities of $\mu_1$ and $\mu_2$ on $E_k$).
Define the standard mollifier
\[
\eta(x) \doteq \begin{cases} C_\eta e^{\frac{1}{|x|^1 -1}}, & |x| < 1 \\
				 0                       , & |x| > 1,
\end{cases}
\]
where $C_\eta > 0$ is such that $\int_{-1}^1 \eta(x) dx = 1$.
For any interval $(a,b)$, $ 0 < a < b$
define 
$h_n:E \rightarrow [0,1]$, $1/m  < 1/2a$ as follows: 
for $x=(j,r) \in E$, $j < k$
$h_n(x) = 0.$ For $x = (k,r)$
\begin{equation}\label{d:hm}
h_n(x) = n\int_{-1}^1 \eta(u/n)\indic_{(a,b)}(u+r)dr, j=k.
\end{equation}
For $j> k$ we proceed recursively:
\begin{equation}\label{e:recdefhm}
h_n(j,0) = \int_0^\infty h_n(j-1,r) g(r)dr,~~~  h_n(j,r) = h_n(j,0) n\int_{nr-1}^1 \eta(x/n)dx,
\end{equation}
where we write $h_n(j,0)$ instead of $h_n( (j,0))$ to simplify
notation.
By its definition, $h_n \in {\mathscr S}$ and 
$\lim_{n \rightarrow  \infty} h_n(k,r) =  \indic_{\{(k,r),r \in (a,b))\}}$
almost surely for any measure $\mu \in {\mathscr M}$, this and the
bounded convergence theorem imply
\[
\lim_{n\rightarrow \infty} \int_E h_n(x) \mu(dx) = \int_a^b m(k,r)dr.
\]
This proves that functions of the form $h_n$ and (therefore the class ${\mathscr S}$ containing them) 
is a separating class for measures in ${\mathscr M}$. 

It remains to show that $h_n \in {\mathscr D}({\mathfrak A})$. 
The following three conditions
for this are listed in Subsection \ref{ss:generator}:
1) $h_n$ must be absolutely
continuous, this follows from its definition \eqref{d:hm} and
\eqref{e:recdefhm}
2) $h_n$ must satisfy 
$h_n((j,0)) = \int_0^\infty h_n(j-1,r) g(r)dr$; this again
holds by definition and 3) ${\mathfrak B}h_n \in L_1^{loc}(X)$; this
follows from the fact that $h_n$ is bounded.
\end{proof}

\begin{proposition}\label{p:proofstationary}
 If the stability assumption  \eqref{as:stable} holds then
$\mu^*_1$
is the unique stationary measure of the process $X$. In particular
if $X_0$ has distribution $\mu^*_1$, $X_t$ has the same distribution
for all $t > 0.$
\end{proposition}
\begin{proof}
The uniqueness follows from the uniqueness claim of Proposition 
\ref{p:muexistsunique}.
By \cite[Proposition (34.7), page 113]{davis1993markov}, 
$\mu^*_1$ is the stationary distribution of $X$ if
\begin{equation}\label{e:zeroaction}
\int {\mathfrak A}h(x) \mu^*_1(dx) = 0;
\end{equation}
for a class of functions $h \in {\mathscr D}({\mathfrak A})$ that forms
a separating class for  measures in ${\mathscr M}$ to which
$\mu_1^*$ belongs; 
by Lemma \ref{l:separating}
${\mathscr S} \subset{\mathscr D}({\mathfrak A})$
is such a class. Thus, to prove the proposition it suffices to prove
\eqref{e:zeroaction} for $h \in {\mathscr S}.$
By definition,
\[
\int {\mathfrak A}h(x) \mu^*_1(dx) = 
\frac{1}{\mu^*(E)} \int {\mathfrak A}h(x) \mu^*(dx),
\]
and one can directly work with the measure $\mu^*$ rather than
the normalized $\mu_1^*.$
For any $h \in {\mathscr S}$
\begin{equation}\label{e:limitproof}
\int_E {\mathfrak A}h(x) \mu^*(dx) = 
\lim_{N\rightarrow \infty} \sum_{k=1}^N 
\int_0^\infty \left(-\frac{dh}{dr}(k,r) + \lambda_0f(r)(h(k+1,r)-h(k,r))
\right)m^*(k,r)dr
\end{equation}
We begin by an integration by parts:
\begin{align*}
&~~\sum_{k=1}^N \int_0^\infty\left(-\frac{dh}{dr}(k,r) + \lambda_0f(r)(h(k+1,r)-h(k,r))
\right)m^*(k,r)dr\\
&~~~=
\sum_{k=1}^N\int_0^\infty
\left( \frac{dm^*}{dr}(k,r) - h(k,0)m^*(k,0) + \lambda_0f(r)(h(k+1,r)-h(k,r))
\right)m^*(k,r)dr\\
\intertext{$h(k,0) = \int_0^\infty g(r) h(k-1,r)dr$  
because $h \in {\mathscr D}({\mathfrak A})$, therefore,}
&~~~=
\sum_{k=1}^N\int_0^\infty
\left( \frac{dm^*}{dr}(k,r) - g(r)h(k-1,r)m^*(k,0) + \lambda_0f(r)(h(k+1,r)-h(k,r))
\right)m^*(k,r)dr\\
\intertext{Rearrange the terms in the sum to factor out the common $h(k,r)$:}
&~~~=
\sum_{k=1}^N\int_0^\infty
\left( \frac{dm^*}{dr}(k,r) - g(r)m^*(k+1,0) + 
\lambda_0f(r)(m^*(k-1,r)-m^*(k,r))\right)h(k,r)dr\\
&~~~~~~~~~~~~~~~~~+ \int_0^\infty m^*(N,r) f(r) \lambda_0 h(N+1,r) dr\\
\intertext{Now ${\mathfrak A}^* \mu^* = 0$ implies}
&~~~= \int m^*(N,r) f(r) \lambda_0 h(N+1,r) dr;
\end{align*}
the last integral goes to $0$ with $N$ because 
$\int m^*(N,r)f(r)dr \rightarrow 0$ and $h$ is bounded. Therefore,
the limit on the right side of \eqref{e:limitproof} is $0$.
This proves \eqref{e:zeroaction} and establishes that $\mu^*$ is the unique stationary distribution of
the process $X$.
\end{proof}

The last proposition and Proposition \ref{p:proofstationary} give
\begin{corollary}\label{c:emptysystemprob}
The stationary probability of an empty system in continuous time for a stable $rM/G/1$ queue is
\[
\mu^*_1({\bm 0}) = \frac{\mu^*({\bm 0})}{\mu^*(E)} = \frac{1-\rho}{1-\rho + \lambda_0 \nu}.
\]
\end{corollary}
\subsection{Expected queue length}
As Corollary \ref{c:emptysystemprob} demonstrates, the 
probability of a stable $rM/G/1$ being empty under its continuous time
stationary distribution does not in general 
equal the same probability under its stationary distribution at service
completion or arrival times:
\[
P_{\mu_1^*}(N_t = 0) = \mu_1^*({\bm 0}) = \frac{1-\rho}{1-\rho + \lambda_0 \nu} \neq q(0) = 1-\rho.
\]
Thus, in general, the steady state queue length distribution of
a stable $rM/G/1$ system in continuous
time differs from the same distribution at service completion and
arrival times.
The following proposition gives a formula for
${\mathbb E}_{\mu_1^*}[N_1]$,under $\mu_1^*$, the expected queue length under
the stationary distribution
in continuous time; in general, this quantity will not equal ${\mathbb E}_q[{\mathcal N}_1]$,
the expected queue length under the
stationary distribution at service completion times.
\begin{proposition}\label{p:meanformula}
The expected $rM/G/1$ queue length under its 
continuous time stationary measure equals
\begin{equation}\label{e:meanformula}
{\mathbb E}_{\mu^*}[N_t] = 
\frac{1}{\mu^*(E)}
\left( \lambda_0^2 \int_0^\infty \left( \int_0^x u f(u)  du\right) g(x)dx
+
\left( (1-\rho)+  {\mathbb E}_q[{\mathcal N}_1] \right)\lambda_0 \nu \right)
\end{equation}
where ${\mathbb E}_q[{\mathcal N}_1]$ is the stationary mean queue length at service completion
times whose formula is given in \eqref{e:pk1}.
\end{proposition}
\begin{proof}
The proof proceeds parallel to that of Proposition \ref{p:mu1starE}.
Set
\begin{equation}\label{d:defvarphi}
\varphi(r)\doteq \sum_{k=1}^\infty m^*_{k}(r) k;
\end{equation}
by definition
\begin{equation}\label{e:meanfirststep}
{\mathbb E}_{\mu^*_1}[N_t] = 
{\mathbb E}_{\mu^*}[N_t]/\mu^*(E) =\frac{1}{\mu^*(E)} \int_0^\infty \varphi(r)dr.
\end{equation}
Let us compute
$\int_0^\infty \varphi(r)dr$.
Multiply the first and the second lines of the
balance equation \eqref{e:balance} by $k$, $k=1,2,3,4,...$
and sum over $k$ to get
\begin{equation}\label{e:dervarphi}
\frac{d \varphi}{dr}(r) + 
\lambda_0 f(r) S(r)
+ \left( \mu^*({\bf 0})\lambda_0 + \sum_{k=1}^\infty m^*_{k+1}(0)k \right)g(r) = 0,
\end{equation}
where, as before,
\[
S(r) =\sum_{k=1}^\infty m^*_k(r) = S(r) = \lambda_0 G(r);
\]
the last equality follows from \eqref{e:valueofS0}.
\eqref{e:setmu0} and \eqref{e:expectedlength1}
simplify the terms in paranthesis in \eqref{e:dervarphi}
to
\[
\left( \mu^*({\bf 0})\lambda_0 + \sum_{k=1}^\infty m^*_{k+1}(0)k \right)
= 
\lambda_0 \left( (1-\rho)+   {\mathbb E}_q[{\mathcal N}_1] 
\right),
\]
where ${\mathbb E}_q[{\mathcal N}_1] $ is the stationary mean queue length at service completions.
Then, the unique solution of \eqref{e:dervarphi} vanishing at $\infty$ is
\[
\varphi(r) =  
\lambda_0^2 \int_r^\infty f(u) G(u)du
+
\left( (1-\rho) + {\mathbb E}_q[{\mathcal N}_1]  \right) \lambda_0 G(r).
\]
Integrating the last display over $r$ over $[0,\infty]$ yields
\begin{align*}
\int_0^\infty \varphi(r)dr 
&= \lambda_0^2 \int_0^\infty \int_r^\infty f(u) G(u)du dr
+
\left( (1-\rho)+  {\mathbb E}_q[{\mathcal N}_1]  \right)\lambda_0 \nu\\
&= \lambda_0^2 \int_0^\infty \left( \int_0^x u f(u)  du\right) g(x)dx
+
\left( (1-\rho)+  {\mathbb E}_q[{\mathcal N}_1]  \right)\lambda_0 \nu.
\end{align*}
This and \eqref{e:meanfirststep} give \eqref{e:meanformula}.
\end{proof}
\subsection{Average arrival rate}\label{ss:averagear}
The random variable
\[
{\mathcal A}_{n} \doteq {\mathcal N}_{n+1} - {\mathcal N}_n + 1,
\]
represents the number of arrivals to the $rM/G/1$ system between the
$n^{th}$ and $(n+1)^{st}$ service completions. It follows from \eqref{e:conddistinc}
that its conditional distribution given in ${\mathcal N}_n$ is
\begin{equation}\label{e:conddistinc}
P( \ {\mathcal A}_{n} = j | {\mathcal N}_n)
= \begin{cases} p(j), &{\mathcal N}_n > 0 \\
	 p(j+1), &{\mathcal N}_n =0,
\end{cases}
\end{equation}
It follows from the Markov property of ${\mathcal N}$
that $({\mathcal N},{\mathcal A})$
is a Markov chain and is stationary whenever ${\mathcal N}$ is, with
stationary distribution
\[
P({\mathcal A}_\infty = j, {\mathcal N}_\infty =  k) = 
 \begin{cases} p(j)q(k), & k >  0 \\
	 p(j+1)q(0), & k =0.
\end{cases}
\]
Then by the ergodic theorem for stable Markov chains
\begin{align}\label{e:ergodiclim1}
\lim_{n\rightarrow \infty} 
\frac{1}{n} \sum_{n=1}^\infty {\mathcal A}_n  &= 
q(0)\left( 1 + \sum_{j=1}^\infty j p(j)\right) + 
(1-q(0))\left( \sum_{j=1}^\infty j p(j) \right) \notag \\
&= q(0) + \rho = 1-\rho + \rho = 1.
\end{align}

Define the interservice time
\[
\tau_n \doteq S_{n+1} - S_n;
\]
similar to the sequence ${\mathcal A}_n$ the distribution of
$\tau_n$ is completely determined by ${\mathcal N}$ with the
following conditional distribution:
\begin{equation}\label{e:conddistinc}
P( \tau_n > t | {\mathcal N}_n)
= \begin{cases} G(t), &{\mathcal N}_n > 0 \\
\int_0^{\infty} \lambda_0 e^{-\lambda_0 s} G((t-s)^+)ds, &{\mathcal N}_n =0,
\end{cases}
\end{equation}
where the second distribution is the convolution of $g$ and the exponential
distribution with rate $\lambda_0$ (this is the distribution
of the sum of a service time and the first arrival time to the system).
The process
$({\mathcal N}, \tau)$ is stable whenever ${\mathcal N}$ is, with the
stationary distribution
\[
P(\tau_\infty  > t, {\mathcal N}_\infty =  k) = 
 \begin{cases} G(t) q(k),&  k > 0, \\
\left(\int_0^{\infty} \lambda_0 e^{-\lambda_0 s} G((t-s)^+)ds\right) q(0),& k=0.
\end{cases}
\]
The law of large numbers for Markov
chains \cite[Theorem 17.0.1, page 422]{meyn2012markov} implies
\begin{align}\label{e:ergodiclim2}
\lim_{n\rightarrow \infty} 
\frac{1}{n} \sum_{n=1}^\infty \tau_n =
\lim_{n\rightarrow \infty} 
\frac{1}{n} S_n  &= 
q(0)\left(\frac{1}{\lambda_0} + \nu\right) + (1-q(0)) \nu\notag\\
& = q(0)\frac{1}{\lambda_0}  + \nu = (1-\rho)\frac{1}{\lambda_0} + \nu.
\end{align}
\begin{proposition}\label{p:arrivals}
Let $A_t$ denote the number of arrivals to an $rM/G/1$ queue up to time
$t$. Then the ergodic average arrival rate to the $rM/G/1$ system equals
\begin{equation}\label{e:arrivalrate}
\lim_{t\rightarrow \infty} \frac{A(t)}{t} = \alpha \doteq
\frac{\lambda_0}{1- \rho + \lambda_0\nu} = \frac{\lambda_0}{\mu^*(E)}.
\end{equation}
\end{proposition}
\begin{proof}
\begin{equation}\label{e:basiclimit}
\lim_{n\rightarrow \infty} 
\frac{A_{T_n}}{T_n} = 
\frac{A_{T_n}/n}{T_n/n} = 
\frac{\lambda_0}{1- \rho + \lambda_0\nu}
\end{equation}
follows from \eqref{e:ergodiclim1} and \eqref{e:ergodiclim2}.
For any other sequence $t_m \nearrow \infty$ we know that
there exists a sequence $n_m$ with $T_{n_m} < t_m <  T_{n_m+1}$.
Borel Cantelli Lemma and that ${\mathcal A}_n$ has finite moments 
independent of $n$ imply
\begin{equation}\label{e:errorterm}
\lim_{n\rightarrow} \frac{{\mathcal A}_n}{T_n} = 0.
\end{equation}
It follows from the monotonicity of $T_n$ and $A_n$ that
\[
\frac{A_{T_{n_m}}}{T_{n_m+1}} \le \frac{A_{t_m}}{t_m} \le
\frac{A_{T_{n_m+1}}}{T_{n_m}}.
\]
This, \eqref{e:errorterm} and \eqref{e:basiclimit} imply 
\eqref{e:arrivalrate}.
\end{proof}

\subsection{Average sojourn and waiting time}
Let ${\varsigma}_k$ be the sojourn time (the total amount of time
spent) of the $k^{th}$ customer arriving to the
system.
Little's law is the following statement
\begin{equation}\label{e:littleslaw}
\lim_{n\rightarrow \infty} \frac{1}{n} \sum_{k=1}^n \varsigma_k =
\frac{\lim_{t\rightarrow} N_t/t}{\lim_{t\rightarrow \infty} A_t/t}.
\end{equation}
The classical proof of this result outlined in
\cite{little2008little} depends on the distribution of $X$
only to the following extent:
that $N$ represents the number of customers in a single server queueing
system and that the ergodic limits related to $N$ and $A$; 
the existence of the ergodic
limits follow from the stationarity of $N$ 
(see, e.g., \cite[Theorem 1.6.4, page 50]{baccelli2012palm}) and
Proposition \ref{p:arrivals} above.
Therefore, the classical proof requires no change for the current setup.
For the $rM/G/1$ system,
\eqref{e:littleslaw} and Proposition \ref{p:arrivals} give
the following formula for the average sojourn time:
\begin{align}\label{e:averagesojourntime}
\lim_{n\rightarrow \infty} \frac{1}{n} \sum_{k=1}^n \varsigma_k &=
\frac{{\mathbb E}_{\mu^*_1}[N_t]}{\alpha} \notag\\
&= 
\lambda_0 \int_0^\infty \left( \int_0^x u f(u)  du\right) g(x)dx
+
\left( (1-\rho)+  {\mathbb E}_q[{\mathcal N}_1]  \right) \nu.
\end{align}
This gives the following formula for the average waiting time
\begin{align}
\omega &\doteq \label{d:omega}
\lambda_0 \int_0^\infty \left( \int_0^x u f(u)  du\right) g(x)dx
+
\left( (1-\rho)+  {\mathbb E}_q[{\mathcal N}_1]  \right) \nu - \nu\\
&=\lambda_0 \int_0^\infty \left( \int_0^x u f(u)  du\right) g(x)dx
+
\left(   {\mathbb E}_q[{\mathcal N}_1]  - \rho \right) \nu,\notag
\end{align}
where ${\mathbb E}_q[{\mathcal N}_1]$ can be computed with formula \eqref{e:pk1}.

\subsection{Connection to $M/G/1$ queue in continuous time}\label{ss:MG1c2}
By Corollary \ref{c:MG1eq1} we know that the embedded
random walk  ${\mathcal N}$ at service completions
of the $rM/G/1$ queue has identical dynamics to
that of an $M/G/1$ queue with constant rate $\lambda_0$ and sequence of
 service times $\{F(\sigma_1), F(\sigma_2),....\}$- which implies
that that these systems have the same stationary measures at
service completions. 
Then a natural question is whether there is  a similar correspondence
between the continuous time stationary distributions.
When $f$ takes the value $0$ over a nonzero interval
its integral $F$ becomes not-invertible.
Because of this, in general, the continuous time stationary distribution
of the $M/G/1$ system cannot completely be mapped to that of the
$rM/G/1$ system (remember that $\sigma_i$ has density $g$; when $f=0$
is allowed $F(\sigma_i)$ may have no density and
$F(\sigma_i)$ may have compact support even when $\sigma$  takes values in
all of ${\mathbb R}_+$).
However, for $f > 0$ an exact mapping between
the stationary measures is possible; the details follow.

Assuming $ f > 0 $ implies $F(r)= \int_0^r f(u)du$ is strictly
increasing.  Let $H$ denote its inverse function; that $F$ is differentiable
implies the same for $H$ and the inverse function has the derivative
\begin{equation}\label{e:derH}
\frac{dH}{ds}(s) = \frac{1}{f(H(s))}.
\end{equation}
Define
\[
\bar{g}(s) \doteq g(H(s))\frac{dH}{ds}(s) = \frac{g(H(s))}{f(H(s))} , s > 0.
\]
For $f > 0$,
the change of variable formula of calculus implies that
$F(\sigma_i)$ has density $\bar{g}.$
The same formula allows one to rewrite 
the operator ${\mathfrak A}^*$ defining the balance equations
of the $rM/G/1$ system
in the $s = F(r)$ variable thus:
\begin{equation}\label{e:defbarA}
\bar{\mathfrak A}^*(\mu)(x) = 
\begin{cases}
\frac{d}{ds} m(k,s)  + \lambda_0 (m(k-1,s) - m(k,s)) + m(k+1,0)\bar{g}(s), k > 1, s > 0,\\
\frac{d}{ds} m(1,s) + \bar{g}(s)\lambda_0 \mu({\bm 0})  + 
\bar{g}(s)m(2,0) 
-  \lambda_0  m(1,s), s > 0,\\
\mu({\bm 0}) \lambda_0 -m(1,0),
\end{cases}
\end{equation}
$\mu \in {\mathscr M}.$
The equation
\begin{equation}\label{e:balance2}
\bar{\mathfrak A}^*(\mu) = 0
\end{equation}
is the balance equation of the $M/G/1$ system with rate $\lambda_0$
and service density $\bar{g}.$ Let us denote the continuous time
process representing this $M/G/1$ system by $\bar{X}$ (which
can be written in the PDP
framework employed in Section \ref{s:construct}).
The relation between
the solution of \eqref{e:balance2}
and the solution of the balance equation \eqref{e:balance} is
given in the following proposition.
\begin{proposition}\label{l:frommubartomu}
Assume $f > 0.$
Let $\mu^*$ be the solution of \eqref{e:balance} given in Proposition
\ref{p:muexistsunique}.
Then $\bar{\mu}^* \in {\mathscr M}$ defined by $\bar{\mu}^*({\bm 0})\doteq
\bar{\mu}^*({\bm 0})$ and by the densities
$\bar{m}^*(k,s) \doteq m^*(k,H(s))$ on $E_k$, $k=1,2,3,...$ solves
\eqref{e:balance2} and does so uniquely up to the choice of
$\bar{\mu}^*({\bm 0}).$ Furthermore, if the stability condition
\ref{as:stable} holds and $\bar{\mu}^*({\bm 0})$ is set to  $1-\rho$ we
 have $\bar{\mu}^*(E) =1$ and $\bar{\mu}^*$ is the unique continuous time
stationary measure of $\bar{X}.$
\end{proposition}
\begin{proof}
${\mathfrak A}^*(\mu^*) = 0
\Rightarrow
\bar{\mathfrak A}^*(\bar\mu^*) = 0 
$
follows from the chain rule.
The uniqueness claim follows from the linearity of \eqref{e:balance2}.
That $\bar{\mu}^*(E) = 1$ under the assumptions \eqref{as:stable} and
$\bar{\mu}^*({\bm 0}) =1-\rho$ follows from the
following observation:
\[
\bar{\mu}^*(E_k) = \int_0^\infty \bar{m}^*(k,s)ds = 
\int_0^\infty m^*(k,r) f(r)dr = q(k), k > 0;
\]
the first equality follows from the change of variable $r=H(s)$
and the last equality follows from \eqref{e:integralrepbmk0} 
and \eqref{e:exactmultiplier}.
That $\bar{\mu}^*$ is the stationary measure of $\bar{X}$ is
proved exactly as in the proof of Proposition \ref{p:proofstationary}.
\end{proof}

\section{Illustration}\label{s:example}
\begin{figure}[h]
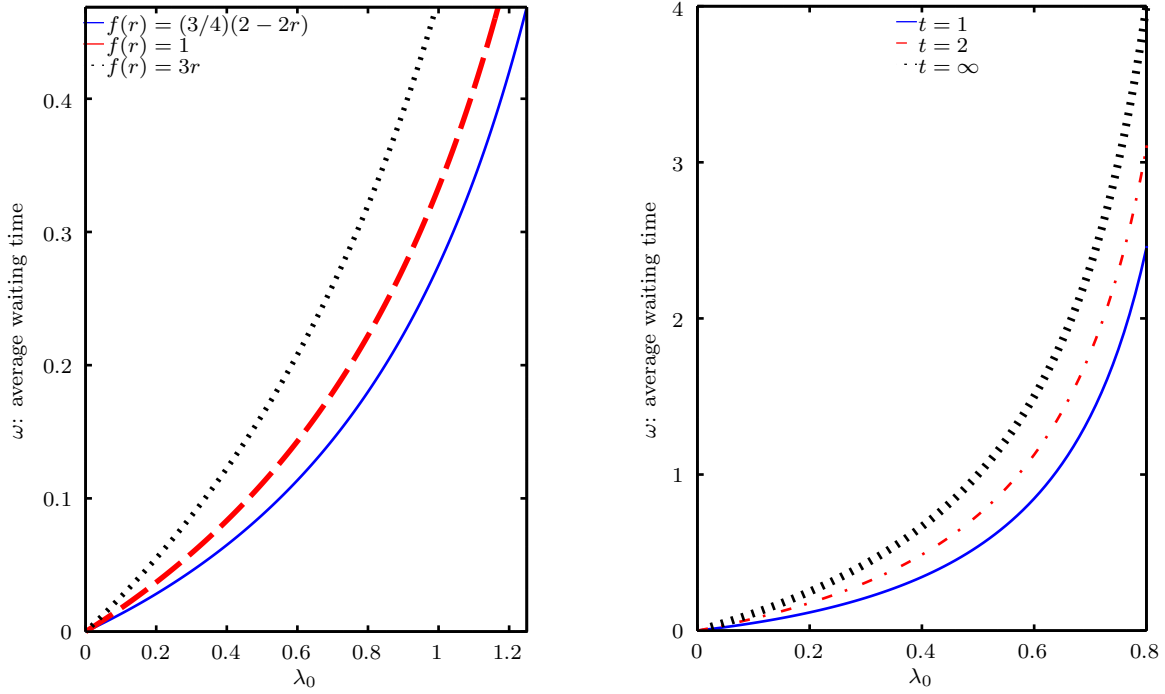

\centering
\hspace{-0.5cm}
\begin{subfigure}{0.48\textwidth}
\input{graph1}
\caption{$g(r) =\indic_{(0,1)}(r)$}
\label{f:exampleuni}
\end{subfigure}
\hspace{0.5cm}
\begin{subfigure}{0.48\textwidth}
\input{graph2}
\caption{$g(r) = e^{-r}$, $f(r) = \frac{1}{1-e^{-t}}\indic_{(0,t)}(r)$}
\label{f:examplex}
\end{subfigure}
\caption{Average waiting time as a function of the average arrival rate for different reshaping functions}\label{figure1}
\end{figure} 
Let us now observe the consequences of the results derived
in the previous section over two examples.
Figure \ref{f:exampleuni} shows
the average waiting times for three $rM/G/1$ systems 
with uniformly distributed service time on the interval $[0,1]$,
as a function
of the arrival rate $\lambda_0.$
We consider three cases for the reshaping function $f$; increasing, constant  
and decreasing in $r$: $f(r) =\frac{3}{4}(2-2r) \indic_{(0,1)}(r)$, 
$f(r)= \indic_{(0,1)}(r)$ and $f(r) = 3r \indic_{(0,1)}(r)$
(the constant case corresponds to the $M/U/1$ queue). 
All of these reshape functions
$f$ satisfy \eqref{e:l0unchanged}, therefore they all have
the same average arrival rate $\alpha = \lambda_0$, 
utilization $\rho = \lambda_0/2$
and empty system probability $\mu^*({\bm 0}) = q(0) = 1-\rho.$ 
Moreover for all of these reshape functions $f$ and the assumed
system parameters, the formula \eqref{d:omega} for the average
waiting time has simple explicit expressions, 
for $f(r) =\frac{3}{4}(2-2r) \indic_{(0,1)}(r)$,
\[
\omega = \omega_1 = \frac{\lambda_0}{8} + \frac{3\lambda_0^2}{40(1-\rho)},
\]
for $f(r)= \indic_{(0,1)}(r)$ (this is the $M/U/1$ case)
\[
\omega = \omega_2 = \frac{\lambda_0}{6} + \frac{\lambda_0^2}{12(1-\rho)}.
\]
and for $f(r) = 3r \indic_{(0,1)}(r)$,
\[
\omega = \omega_3 = \frac{\lambda_0}{4} + \frac{9\lambda_0^2}{80 (1-\rho)}.
\]
We note $\omega_1 < \omega_2 < \omega_3$ for all $\lambda_0$ such that
$\rho < 1$: i.e., pushing arrivals towards service completions (while
keeping the average arrival rate constant) reduces average waiting times.
Figure \ref{f:exampleuni} shows the graphs of these functions as
$\lambda_0$ varies.

Let us now consider an example in which the service time is exponentially
distributed mean $\nu = 1$ and the reshaping function 
$f(r) = (1-e^{-t})^{-1}\indic_{0\leq r \leq t}$;  this function restricts
arrivals to the last $t$ units of time of service and it satisfies
\eqref{e:l0unchanged}. The average waiting time $\omega$ of \eqref{d:omega}
reduces for this case to
\[
\omega(t,\lambda_0)
= 
\lambda_0 \frac{1 -e^{-t}(t+1)}{1-e^{-t}}\left( 1 + 
\frac{\lambda_0^2}{(1-e^{-t})(1-\rho)}\right);
\]
$\omega$ is increasing in $t$, i.e., once again, concentrating
arrivals near service completions (while keeping the average arrival rate constant) reduces the average waiting time.
Figure \ref{f:examplex}, 
shows the graph of $\omega(t,\cdot)$ for $t=1,2$ and $t=\infty$
(the last corresponds to an $M/M/1$ queue).

We conclude this section with the following observation from
our second example: set
$t=3\nu = 3$ in the last example, i.e., we restrict arrivals to the
interval $[0,3\nu]$, where $\nu$ is the mean service time. For $\lambda_0 = 0.7$,
the system's utilization is $\rho = 0.7$ and the corresponding average
waiting time turns out to be $\omega(3,0.7) =  1.6041$; the same waiting
time for the same parameter values but without reshaping is 
$\omega(\infty,0.7)=1.84$. Thus, this not so heavy reshaping reduces 
average waiting time by $13\%.$
\section{Conclusion}\label{s:conclusion}
Let us comment briefly on possible future research. We have assumed that
the service time distribution has a density $g$. The analysis at service completions
doesn't depend on this assumption and the results of Subsection \ref{ss:completion}
and Section \ref{s:servcomp} continue to hold without change 
when $\sigma_i$ doesn't have a density.
The analysis of Section \ref{s:stationarycont} does make use of the 
assumption that
$\sigma_i$ has a density but the resulting performance measure formulas 
(average queue length, probability of an empty system, average waiting time)
remain meaningful even when $\sigma_i$ doesn't have a density and one
expects these results to hold under general service distributions. 
One simple method of extending our analysis to the general case would be first a 
smooth approximation of the given service distribution and then taking weak limits. 
The details of such an argument could be given in future work. The special case of a deterministic constant 
service time case can be directly handled by appropriate modifications of the balance equation and
our arguments based on it.

A natural question is the convergence of the distribution of $X_t$
to the stationary distribution $\mu^*$. As one of the referees pointed out,
one way to establish this with precise rates of convergence would be to apply
the approach of \cite{lund1996computable} 
based on coupling (at the first hitting time to ${\bm 0}$)
and monotonicity arguments.  
Future research could attempt to give details of this.

In many situations, one may only have an estimate
of the remaining service time (rather than the ability to
directly observe it, as assumed in the current work). One possible
future work is the modeling and analysis of such a setup.
We think that, given the possible applications 
in call centers,
another natural direction is the treatment of many servers. Instead of allowing
the rate to depend directly on the remaining service times of all of the
servers a possibility is to allow it to depend on 
a function of them
(e.g., their minimum or an estimate of it). Finally, it may also be of 
interest to apply the approach used in the present
article to models
where the arrival and service rates depend
on the queue length as well as the remaining service time.

\begin{singlespace}
\bibliography{1}

\providecommand{\bysame}{\leavevmode\hbox to3em{\hrulefill}\thinspace}
\providecommand{\MR}{\relax\ifhmode\unskip\space\fi MR }
\providecommand{\MRhref}[2]{%
  \href{http://www.ams.org/mathscinet-getitem?mr=#1}{#2}
}
\providecommand{\href}[2]{#2}
\begin{thebibliography}{10}

\bibitem{aksin07}
O.Z. Ak\c{s}in, M.~Armony, and V.~Mehrotra, \emph{The {M}odern {C}all-{C}enter:
  A {M}ulti-{D}isciplinary {P}erspective on {O}perations {M}anagement
  {R}esearch}, Production and Operations Management \textbf{16} (2007),
  665--688.

\bibitem{asmussen2008applied}
Soren Asmussen, \emph{Applied probability and queues}, vol.~51, Springer
  Science \& Business Media, 2008.

\bibitem{baccelli2012palm}
Fran{\c{c}}ois Baccelli and Pierre Br{\'e}maud, \emph{Palm probabilities and
  stationary queues}, vol.~41, Springer Science \& Business Media, 2012.

\bibitem{bekker2004queues}
Ren{\'e} Bekker, Sem~C Borst, Onno~J Boxma, and Offer Kella, \emph{Queues with
  workload-dependent arrival and service rates}, Queueing Systems \textbf{46}
  (2004), no.~3-4, 537--556.

\bibitem{davis1993markov}
Mark~HA Davis, \emph{Markov models \& optimization}, vol.~49, CRC Press, 1993.

\bibitem{dshalalow1991single}
Jewgeni Dshalalow, \emph{On single-server closed queues with priorities and
  state dependent parameters}, Queueing Systems \textbf{8} (1991), no.~1,
  237--253.

\bibitem{gans03}
N.~Gans, G.~Koole, and A.~Mandelbaum, \emph{Telephone {C}all {C}enters:
  {T}utorial, {R}eview, and {R}esearch {P}rospects}, Manufacturing \& Service
  Operations Management \textbf{5} (2003), 73--141.

\bibitem{kleinrock75}
Leonard Kleinrock, \emph{Queueing {S}ystems, {T}heory}, vol.~I, A
  Wiley-Interscience Publication, 1975.

\bibitem{knessl1987busy}
Charles Knessl, BJ~Matkowsky, Zeev Schuss, and Charles Tier, \emph{Busy period
  distribution in state-dependent queues}, Queueing Systems \textbf{2} (1987),
  no.~3, 285--305.

\bibitem{knessl1987markov}
\bysame, \emph{A markov-modulated m/g/1 queue i: stationary distribution},
  Queueing Systems \textbf{1} (1987), no.~4, 355--374.

\bibitem{knessl1987markov2}
\bysame, \emph{A markov-modulated m/g/1 queue ii: Busy period and time for
  buffer overflow}, Queueing Systems \textbf{1} (1987), no.~4, 375--399.

\bibitem{kumar2009enhanced}
Dinesh Kumar, Li~Zhang, and Asser Tantawi, \emph{Enhanced inferencing:
  Estimation of a workload dependent performance model}, Proceedings of the
  Fourth International ICST Conference on Performance Evaluation Methodologies
  and Tools, ICST (Institute for Computer Sciences, Social-Informatics and
  Telecommunications Engineering), 2009, p.~47.

\bibitem{legros16}
Benjamin Legros, Oualid Jouini, and Ger Koole, \emph{Optimal scheduling in call
  centers with a callback option}, Performance Evaluation \textbf{95} (2016),
  1--40.

\bibitem{little2008little}
John~DC Little and Stephen~C Graves, \emph{Little's law}, Building intuition,
  Springer, 2008, pp.~81--100.

\bibitem{lund1996computable}
Robert~B Lund, Sean~P Meyn, Richard~L Tweedie, et~al., \emph{Computable
  exponential convergence rates for stochastically ordered markov processes},
  The Annals of Applied Probability \textbf{6} (1996), no.~1, 218--237.

\bibitem{meyn2012markov}
Sean~P Meyn and Richard~L Tweedie, \emph{Markov chains and stochastic
  stability, second edition}, Cambridge University Press, 2012.

\bibitem{pang14}
Guodong Pang and Ohad Perry, \emph{A logarithmic safety staffing rule for
  contact centers with call blending}, Management Science \textbf{61} (2014),
  no.~1, 73--91.

\bibitem{perry2013duality}
D~Perry, W~Stadje, S~Zacks, et~al., \emph{A duality approach to queues with
  service restrictions and storage systems with state-dependent rates}, Journal
  of Applied Probability \textbf{50} (2013), no.~3, 612--631.

\end{thebibliography}
\end{singlespace}
\end{document}